\documentclass[10pt]{amsart}
\usepackage{graphicx}
\usepackage{amsmath,amsfonts,amsthm,amssymb,graphics,color}

\newtheorem{theorem}{Theorem}
\newtheorem{prop}{Proposition}
\newtheorem{lemma}{Lemma}

\newcommand{\pa}{\partial}

\newcommand{\W}{W^{1,4} _h (\Omega)}

\newcommand{\bx}{{\bf x}}
\newcommand{\by}{{\bf y}}

\newtheorem{cor}{Corollary}

\theoremstyle{remark}
\newtheorem{remark}{Remark}
\theoremstyle{definition}
\newtheorem{defn}[equation]{Definition}

\numberwithin{equation}{section}

\newcommand{\cK}{\mathcal{K}}

\newcommand{\cB}{{\mathcal{B}}}

\newcommand{\R}{\mathbb{R}}

\newcommand{\cC}{\mathcal{C}}
\newcommand{\cL}{\mathcal{L}}

\newcommand{\n}{\nabla}

        \definecolor{pink}{rgb}{1,0,1}

\begin{document}
\title[Optimally Transported Erosion]{Mathematical Models for Erosion and the Optimal Transportation of Sediment}

\author[Bjorn Birnir]{Bjorn Birnir}
\address{Department of Mathematics, South Hall 6607, University of California, Santa Barbara, CA 93106.} \email{birnir@math.ucsb.edu}

\author[Julie Rowlett]{Julie Rowlett}
\address{Max Planck Institut f\"ur Mathematik \\ Vivatgasse 7 \\ D-53111 Bonn}
\curraddr{Chalmers University and the University of Gothenburg \\ Mathematical Sciences \\ 41296 Gothenburg Sweden} 
\email{julie.rowlett@chalmers.se} 


\begin{abstract}
 We investigate a mathematical theory for the erosion of sediment which begins with the study of a non-linear, parabolic, weighted 4-Laplace equation  on a rectangular domain corresponding to a base segment of an extended landscape.  Imposing natural boundary conditions, we show that the equation admits entropy solutions and prove regularity and uniqueness of weak solutions when they exist.  We then investigate a particular class of weak solutions studied in previous work of the first author and produce numerical simulations of these solutions.  After introducing an optimal transportation problem for the sediment flow, we show that this class of weak solutions implements the optimal transportation of the sediment.  
\end{abstract}

\maketitle

\section{Introduction}   
The continuing evolution of the surface of the earth poses a challenging and fascinating modeling problem.  The earth's surface is composed of many substances:  soil, sand, vegetation, and different types of rocks.  Its surface is further complicated by topography which continues to change over time due to tectonic uplift and earthquakes. Due to the complexity of most landsurfaces and the instability of some it has taken many years to develop models. The theory of fluvial landscape evolution began with geological surveys such as \cite{gil} and \cite{dav} which were developed into geological models such as \cite{hor} and \cite{mac}.
The investigations performed more recently fall into three groups{:}  (1) empirical investigations of fluid phenomena, (2) computational investigations of discrete models and (3) investigations of continuous model, or partial differential equations, of surface evolution and channelization. The first group includes the field observations \cite{sch56} on the badlands of the Perth Amboy and the flume \cite{sch87} and artificial stream \cite{bmvp108,bmvp208} experiments that have given deep insight into channelized drainage. The second group has produced remarkable simulations of evolving channel networks{;} see \cite{will91a, will91b}, \cite {how}, \cite{tuck97} and \cite{rr}. The third group has lead to an increasing understanding of the physical mechanisms that underlie erosion and channel formation{;} see \cite{s1},\cite{sb}, \cite{lu74}, \cite{roth89}, \cite{low}, \cite{kra92}, \cite{loe94}, \cite{sm}, \cite{izumi95a,izumi95b,izumi00a,izumi00b,izumi06},\cite{bms1,bms2,bms3}, \cite{sim06}, \cite{wbb}, \cite{stfl}, \cite{fow07}, \cite{bb}, \cite{s2}.
 
In \cite{bms1} and \cite{bms2}, a family of two partial differential equations were introduced based on the conservation of water and sediment.  These equations describe a transport-limited process \cite{how} in which sediment moves in the same direction in which the surface water flows. The transport-limited case models situations found in badlands and deserts where all the sediment can be transported away if a sufficient quantity of water is available. The detachment-limited case is the other extreme;  see \cite{izumi95a,izumi95b,izumi00a,izumi00b,izumi06} which model a situation where the surface is covered with rock that must weather before the resulting sediment can be transported away. Different models are required for the latter situation.

 Our focus is the equations developed and studied by the first author and his collaborators in \cite{bms1}, \cite{bms2}, \cite{scaling}, \cite{stfl}. These equations improve the original model in \cite{sb} by including a pressure term (in addition to the gravitational and friction terms) that prevents water from accumulating in an unbounded manner in surface concavities; see \cite{low}. They present a representation of the free water surface in a diffusion analogy approximation to the St. Venant equations; see \cite{wei79}.

\begin{figure} \label{surface} \begin{center} \includegraphics[height=5.0in]{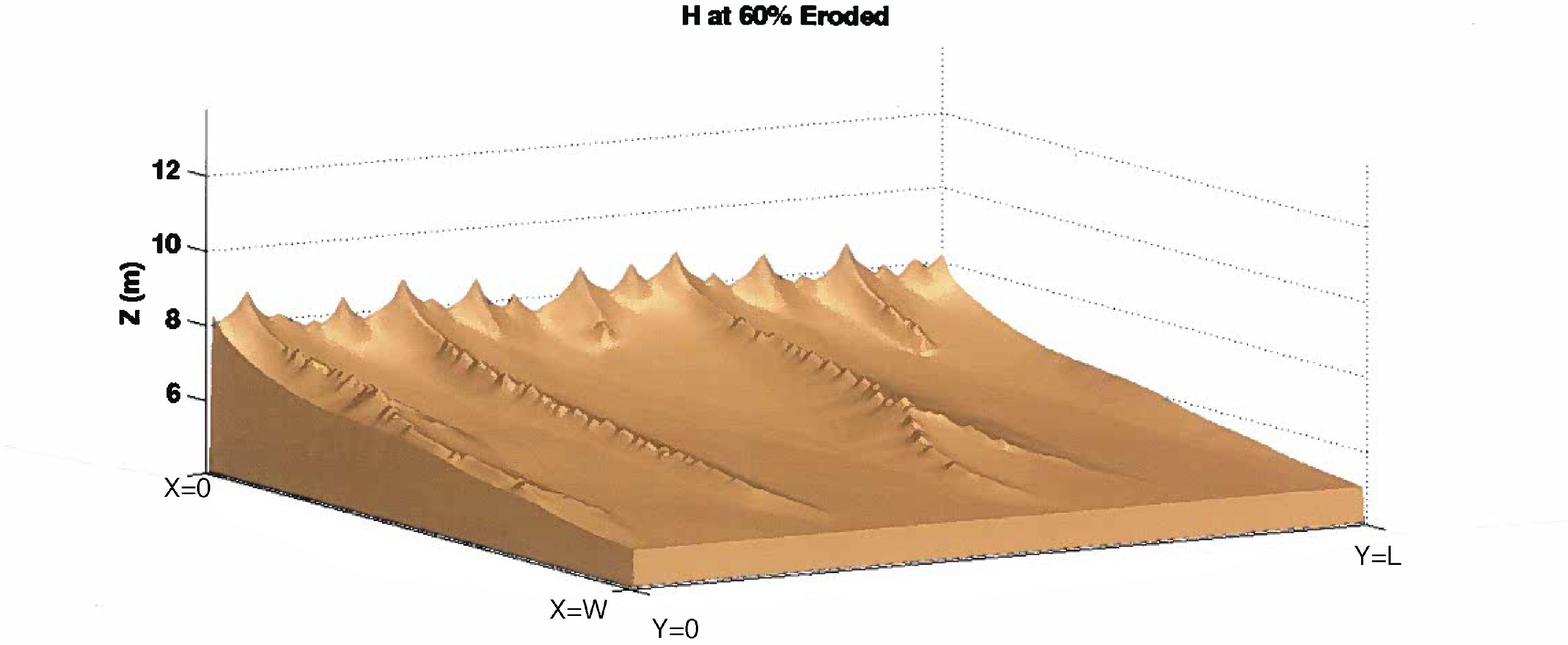} \caption{A desert landscape consisting of a pattern of valleys separated by mountain ridges, from \cite{scaling}.} \end{center} \end{figure}

The analysis of these equations has so far been mostly numerical, and simulations show a striking time evolution that seems to be similar to the evolution of realistic landscapes; see Figure 1 taken from \cite{scaling}.  Since the model consists of two equations, one equation for the water flow and the other for the sediment flow, each of which has a different time scale,  to analyze the equation for the sediment flow we will think of the water depth $h$ as a non-negative quantity which has been averaged over many rainfall events on a fast scale.  Numerically a statistically stationary water surface is observed to exist. We will then investigate the equation modeling the sediment flow that is associated with a larger time scale.  

Our first main result is that for integrable initial data, entropy solutions exist and are unique.  In the following theorem, $\Omega$ is a rectangular domain.  

\begin{theorem} \label{th:e}  Let $h$ be a given function which satisfies the assumptions (\ref{eq:h}) and (\ref{eq:Mucken}) given in \S 3.  Then, for any $H_0 \in W^{1,1} (\Omega)$ and any $T>0$, there exists a unique \em entropy solution \em (see (\ref{ent-sol})) to the model equation for the sediment (\ref{eq:toy2}) on $(0, T) \times \Omega$.  
\end{theorem}  

This mathematical theory is important to further study both the deterministic and stochastic aspects of landscape evolution. However, the mathematical theory does not give an immediate interpretation in terms of the physical process and observed phenomena.  Seeking further particulars of the solutions and wishing to demonstrate that they give rise to reasonable models, we were naturally lead to a related problem that interestingly provides the connection to an optimality principle.

The theory of optimal transport began in 1781 with Monge's simple question \cite{monge}:  what is the least expensive way to transport mounds of dirt in order to fill holes?  In the most naive terms, erosion is nature's process of ``moving dirt,'' so one would expect it to be transported optimally in an appropriate sense.  Erosion takes place at each point of the eroding surface, and the eroded sediment is then transported by the river network to a river or lake at the lower boundary of the region.  It turns out that the easiest way of expressing this is in terms of the erosion rate at each point of the surface and the flux of sediment through the boundary of the region. In a period of time this amounts to a layer of sediment being eroded from the surface and transported through the boundary.  
 
Connections between optimal transport and falling sandcones have been previously demonstrated in \cite{evans} using a local $p$-Laplacian evolution equation and in \cite{amrt1} and \cite{amrt2} using a non-local $p$-Laplace evolution equation.  In our work, the model equation for the sediment is a weighted local $p$-Laplacian evolution equation with mixed boundary conditions.    In our next main result, we show that if a solution of the equation satisfies a certain condition (\ref{eq:curl}), then the solution predicts the direction in which the sediment flows when it is optimally transported. 

\begin{theorem} \label{th:ot1}  Assume that for a given function $h$ satisfying (\ref{eq:h}) and $H_0 \in \W$, $H$ is a weak solution of (\ref{eq:toy2}) on $[0,T]$ with initial data given by $H_0$.  
Assume that at $t \in (0,T)$ (\ref{eq:massbal}) and (\ref{eq:out}) are satisfied, and let $\mu$ and $\nu$ be the measures supported on $\Omega$ and defined by (\ref{eq:mn}).  Then, there exists an optimal mass reallocation plan $s: \Omega \to \Omega$, which solves (\ref{monge}), and there exists a function $u$ so that $s$ and $u$ satisfy the equation
\begin{equation} \label{eq:su} \frac{s(\bx) - \bx}{|s(\bx) - \bx|} = - \n u. \end{equation}
Moreover, if $\n H$ is defined a.e. on $\Omega$ and satisfies (\ref{eq:curl}) at time $t$ at a.e. points where $\n H$ is defined and non-zero, then at these points 
\begin{equation}
\label{eq:grad}
 \n u =  \frac{\n H}{|\n H|}.
\end{equation}
In this case the sediment flow implements the optimal transport.  
\end{theorem}

This theorem distinguishes a certain class of \em optimal model solutions \em to the equation.  In \S 4, we provide examples of solutions for which the above theorem holds, and we produce both graphs and numerical simulations of the equation seeded by these solutions.  From the data, one sees that these solutions are a close approximation of the observed mountain ridges and valleys.  Nature should implement a statistically optimal transportation of the sediment, so the above theorem reinforces the observations in \S 4, which show that the model solutions appear to accurately model the landsurface and erosion process.  

\section{The model equations} 
The model equations are based on a conservation principle of water and sediment fluxes over a continuous, erodible surface $z = z(x,y,t)$, and on the advective entrainment and transport of sediment in transport limited conditions as in \cite{how2}.  For a detailed derivation of these equations we refer to \cite{bms2} and \cite{scaling}.  
\begin{equation} \label{eq:toy1}  -\nabla \cdot \left[ \frac{\nabla H}{| \nabla H|^{1/2}} h^{5/3} \right] = R,
\end{equation}
and
\begin{equation} \label{eq:toy2} \frac{\partial H}{\partial t} =\nabla \cdot \left[ \nabla H\ | \n H |^{2}  h^{10 / 3} \right].
\end{equation}
When many simulations are performed and an ensemble average over these simulations is taken, a statistically stationary equilibrium water depth emerges. Based on this numerical evidence we will assume in this paper that a statistically stationary water depth exists and make assumptions on it based on the numerical evidence. We will use this statistically stationary (average) water depth $h$ to study the sediment flow as described by the second equation (\ref{eq:toy2}).

\subsection{Boundary Conditions} 
We use the same boundary conditions as \cite{bms1} and \cite{bms2} to model a ridge defined over a rectangular domain of length $L$ and width $W$,
$$
\Omega = \{(x,y)\in {\Bbb R}^2|\ 0 \leq x \leq W, \; 0 \leq y \leq L \}.  
$$

\begin{figure}
\label{watersurface}
\begin{center}
\includegraphics[height=5.0in]{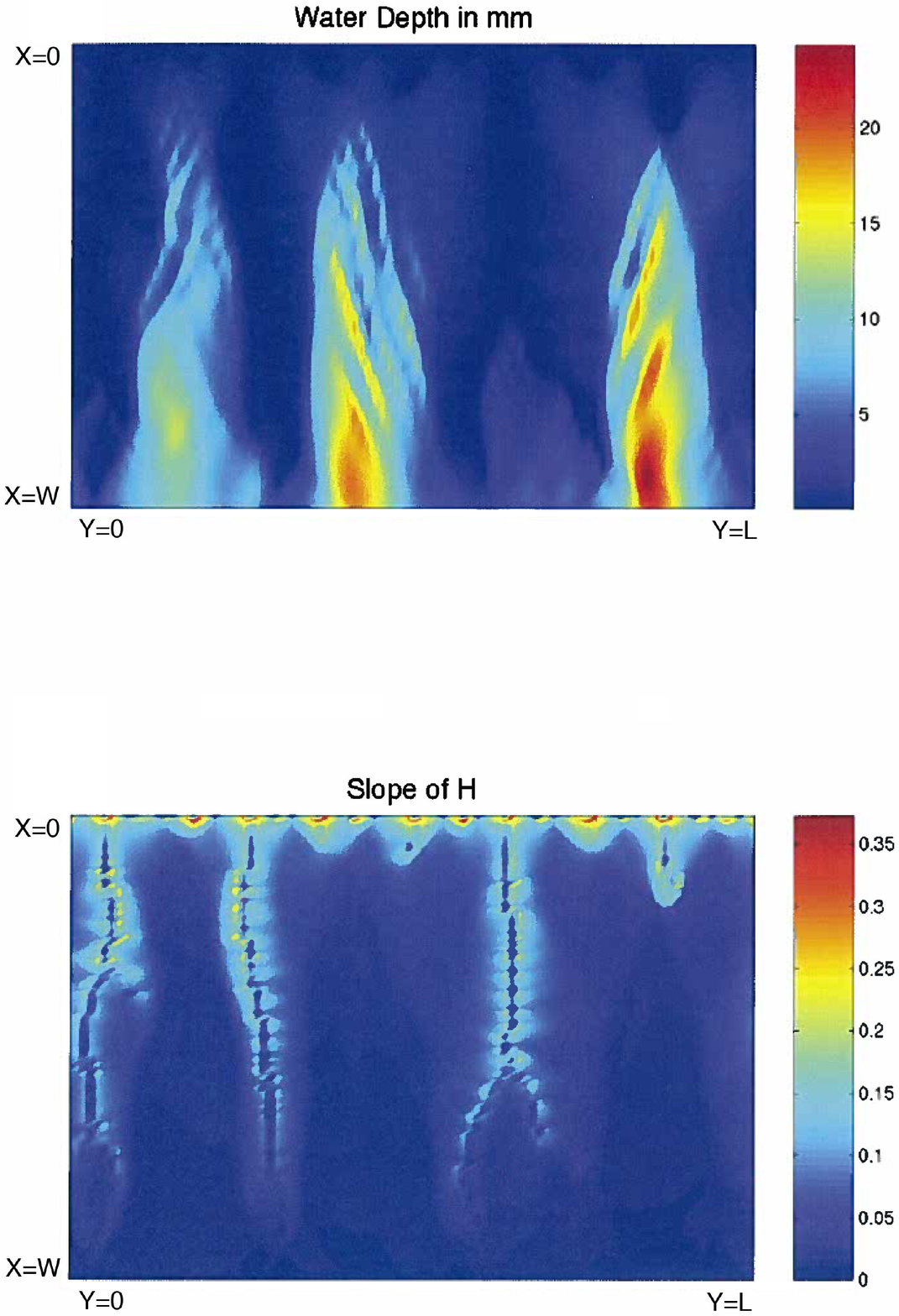}
\caption{The water depth (top) and slope of the water surface (bottom) over the desert landsurface in Figure 1, from \cite{scaling}.}
\end{center}
\end{figure}

The water depth $h$ is shown on the top part of Figure 2, and the gradient of the slope of the water surface $H$ is shown on the bottom part.  The darkest color indicates zero slope, and it is clear that the surface consists of three mountain ridges separating three valleys.  This corresponds to the domain being a finite segment of an extended mountain range, as shown in Figure 2.  It is therefore natural to impose periodic boundary conditions in $y$, by assuming 
$$H(x, 0, t) = H(x, L, t); \quad h(x, 0) = h(x, L).$$ 
We assume that the lower boundary at $x=W$ is a river or a lake that absorbs all the water and sediment.  Therefore, on this boundary we impose the Dirichlet condition for $H$, 
$$H(W, y, t) = 0.$$
This can be taken to be the elevation of the river or lake.  On the top of the slope, at $x=0$ in Figure 2, we assume that there is no sediment flowing over the top of the mountain ridge which corresponds to the Neumann condition for the sediment flux, 
\begin{equation} \label{amrtbc} 
|\n H|^2 h^{10/3} \n H \cdot n = 0 \quad \textrm{ at } x=0,
\end{equation}
where $n$ is the unit outward normal vector.  At the top of the mountain ridges both the slope and the water depth tends to zero. This is indicated by a very dark blue in contrast to the yellow and red at the bottoms of the valleys were the water accumulates.  We shall therefore assume that 
$$h(0, y) = 0.$$ 

In summary, the boundary conditions are as follows 
\begin{eqnarray}
\label{eq:boundary}
h(0,y) &=& 0, \nonumber \\ 
|\n H(0, y, t)|^2 h(0, y)^{10/3} \n H(0, y, t) \cdot n &= & 0, \nonumber\\
h(x,0) &=& h(x,L), \nonumber \\ 
H(x,0,t) &=&H(x,L,t) \nonumber \\ 
H(W,y,t) &=& 0.  
\end{eqnarray}

\section{Solutions of the model equations}
The model equations (\ref{eq:toy1}) and (\ref{eq:toy2}) are highly nonlinear, and no explicit solution of the initial boundary value problem in the previous section exists. The initial surface is unstable, but in spite of this it is possible to solve the two equations (\ref{eq:toy1}) and
(\ref{eq:toy2}) numerically with modern numerical methods.
The first author and his collaborators did this in \cite{bms1}, \cite{bms2}, \cite{scaling}, \cite{stfl} and \cite{wbb}. Thus they gained considerable insight into the properties of the solutions, and the main purpose of this paper is to develop the full nonlinear analysis based on these insights. 

\subsection{Existence and Uniqueness of Entropy Solutions}  
Since our focus is the sediment over a long time scale as described by (\ref{eq:toy2}),  
we will assume that the water depth function $h$ is given, does not depend on time and satisfies 
\begin{equation} \label{eq:h} h \geq 0, \quad h > 0 \textrm{ a. e. on } S, \quad h \in \cL^\infty (\Omega), \end{equation}
where $S \subset \Omega$ is a piecewise smooth domain contained in $\Omega$.  We shall assume further that 
$$h^{-10/9} \in \cL^1 (S).$$ 

The equation (\ref{eq:toy2}) is a weighted $p$-Laplacian evolution equation, with $p=4$.   
Our proof of existence and uniqueness of entropy solutions is based on \cite{amrt3}.  To solve the equation, we introduce a corresponding weighted Sobolev space,  
$$W^{1,4} _h (\Omega) := \left\{ u \in L^4(\Omega) \textrm{ such that } h^{5/6} \frac{\pa u}{\pa x} \in L^4 (S) \textrm{ and } h^{5/6} \frac{\pa u}{\pa y} \in L^4 (S) \right\},$$
where 
\[
\Omega =  [0,W]\times \mathbb{T}^1;\:\:\: x\in [0,W],\; y \in \mathbb{T}^1=[0,L], 
\]
is a cylinder, which we use due to the periodic boundary conditions in $y$ (\ref{eq:boundary}).  
We shall work with a corresponding weighted Sobolev norm on this space, 
$$||u||_{W^{1,4} _h (\Omega)} := \left( \int_\Omega \left( |u|^4 +|\n u|^4 h^{10/3} \right) d\bx \right) ^{1/4}.$$ 
We use $d\bx$ to denote integration with respect to the standard Lebesgue measure on $\R^2$.  We shall also work with the standard Sobolev space, 
$$W^{1,1}  (\Omega) := \left\{ u \in L^1(\Omega) \textrm{ such that } \n u \in L^1(\Omega) \right\}.$$

\begin{defn} A function $H \in W^{1,1} (0, T; L^1 (\Omega))$ is an \em entropy solution \em of (\ref{eq:toy2}) on $(0, T)$ with initial data $H_0 \in W^{1,1} (\Omega)$ if
\begin{align} 
\label{ent-sol} &H(0) = H_0, \nonumber \\
 &T_k (H(t, x, y)) := \sup \{ \inf \{H(t, x, y), k\}, -k \} \in W^{1,4} _h (\Omega) \, \, \forall \, \, k > 0, \nonumber \\
& \int_\Omega \left( H'(t) T_k (H(t) - \phi) + h^{10/3} |\n H(t)|^2 \n H(t) \cdot \n (T_k (H(t) - \phi) ) \right) d\bx \leq 0;  
 \end{align} 
 where the last equation holds for all $\phi \in W^{1,4} _h (\Omega) \cap L^{\infty} (\Omega)$.  
 \end{defn} 

Above, the standard truncation function $T_k (H)$ is equal to $H$ if the value of $H$ lies in $[-k,k]$, and otherwise is equal to $-k$ if $H < -k$ or $k$ if $H > k$.

To prove the existence of entropy solutions we use the non-linear semi-group theory developed in \cite{pbmc1}, \cite{pbmc2}, and \cite{mgc}.  This requires the following technical assumptions on the function $h$, 
\begin{equation} \label{eq:Mucken} \sup \left( \frac{1}{|B|} \int_B \tilde h ^{10/3} d\bx \right) \left( \frac{1}{|B|} \int_B \tilde h ^{-10/9} d\bx \right)^3 < \infty . \end{equation} 
Above, $|B|$ is the volume of the ball $B$, and the supremum is taken over all balls $B \subset \R^2$.  We assume that there exists a function $\tilde h$ on $\R^2$ which satisfies this condition and such that $\tilde h = h$ a.e. on $S$.  This condition implies that $\tilde h^{10/3}$ is in Muckenhoupt's $A_4$ class, and that smooth functions are dense in $W^{1,4} _h (\Omega)$ with respect to the associated norm.  In particular, the following was proven in \cite{amrt3}.  

\begin{prop}[3.2 from \cite{amrt3}]  \label{prop:amrt} 
For any $u \in \W \cap \cL^\infty (\Omega)$ there exists a sequence $\{\phi_n \}_{n=1} ^\infty \subset \cC^\infty (\Omega)$ such that 
$$\lim_{n \to \infty} || \phi_n - u ||_{\W} = 0.$$
\end{prop}

To study the evolution equation, we introduce the associated elliptic problem, 
\begin{align} \label{eprob} 
& u - \n \cdot \left[ h^{10/3} |\n H|^2 \n H\right] = f \textrm{ in } \Omega; \nonumber \\
& h^{10/3} |\n H|^2 \n H \cdot n = 0 \textrm{ at } x=0; \nonumber \\ 
& H(W, y) = 0.   
\end{align} 

The following operator will be used to produce solutions. 
\begin{defn} \label{def:B}
We define the operator $\cB$ on $\cL^1 (\Omega)$ such that $(u, \hat{u}) \in \cB$ if and only if $u \in \W \cap \cL^\infty (\Omega)$, $\hat{u} \in \cL^1 (\Omega)$, and 
$$\int_\Omega h^{10/3} |\n u|^2 \n u \cdot \n v d\bx = \int_\Omega \hat{u} v d\bx \quad \forall \quad v \in \W \cap \cL^\infty (\Omega).$$
\end{defn} 

At the $x=0$ boundary, we have the same condition as \cite{amrt3}.  The periodic boundary conditions in $y$ has been handled by working over the cylinder.  Since the remaining boundary at $x=W$ is the standard Dirichlet condition for $H$, by the assumptions (\ref{eq:h}) and (\ref{eq:Mucken}), Proposition 3.5 of \cite{amrt3} implies that $\cB$ is completely accretive and satisfies the range condition $\cL^\infty (\Omega) \subset R(I + \cB)$.  Moreover, since $h>0$ a.e. on $\Omega$, Proposition 3.6 of \cite{amrt3} shows that the closure of $\cB$ in $\cL^1 (\Omega) \times \cL^1 (\Omega)$ is given by $(u, v) \in \cB$ if $u, v \in \cL^1(\Omega)$, $T_k (u) \in \W$, and 
$$\int_\Omega h^{10/3} |\n u|^2 \n u \cdot \n (T_k (u - \phi)) d\bx \leq \int_\Omega v T_k (u - \phi) d\bx,$$
for all $\phi \in \W \cap \cL^\infty (\Omega)$ and all $k>0$.  Consequently, the non-linear semi-group theory (\cite{mgc}, \cite{pbmc2}) together with the above properties of $\cB$ imply the existence and uniqueness of entropy solutions (c.f. Theorem 3.7 in \cite{amrt3}).  
\qed

\subsection{Weak Solutions} 
Due to the presence of the weight function $h$ in (\ref{eq:toy2}) for arbitrary initial data in $\W$, it is only possible to prove the existence and uniqueness of entropy solutions to (\ref{eq:toy2}).  To see that standard arguments for the existence of weak solutions cannot be applied here, consider the Volterra operator 
$$V_g : \cL^2 (I) \to \cL^2 (I), \quad f \mapsto fg, \quad g \in \cC([0,1]).$$
The spectrum of $V_g$ is $g([0,1])$, so if $g$ vanishes at any point of the interval then $V_g$ is not invertible.  Similarly, depending on the zero set of $h$, $W^{1,4} _h (\Omega)$ may be a proper subspace of $W^{1,4} (\Omega)$, and so the standard arguments which rely on the Sobolev embedding theorem cannot be used to prove existence of weak solutions for arbitrary initial data.  Nonetheless, for applications weak solutions are useful, and we shall construct explicit weak solutions in \S 4, so we conclude this section be demonstrating estimates for weak solutions to (\ref{eq:toy2}).  

\begin{defn} Given $h$ which satisfies (\ref{eq:h}) and (\ref{eq:Mucken}), assume further that $h$ is continuous on $\Omega \setminus h^{-1} (0)$ and satisfies 
\begin{equation} \label{h-zero} h^{-1} (0) \textrm{ is the finite union of piecewise smooth curves.} 
\end{equation} 
Then a \em weak solution \em of (\ref{eq:toy2}) on $(0,T)$ with initial condition $H_0 \in  \W$ is 
$$H \in W^{1,4} _h ((0,T); \Omega),$$
which satisfies 
\begin{align} & H = H_0 \textrm{ a.e. on $\Omega$,}  \quad t=0; \nonumber \\ &\int_\Omega \left( \frac{\pa H}{\pa t} (H - \phi) + h^{10/3} |\n H(t)|^2 \n H(t) \cdot \n (H - \phi ) \right) d\bx = 0,  \quad  t \in (0,T);  \end{align} 
for all $\phi \in \cC^\infty _0 (\Omega \setminus h^{-1} (0))$.  
\end{defn} 

These assumptions can be physically interpreted as follows.  The water height at the exact precipice of a ridge may vanish, and the ridges are modeled by a finite union of piecewise smooth curves, so $h^{-1}(0)$ corresponds to the top of the ridges.  Since the sediment flows down either side of the ridge, the direction of sediment flow is discontinuous at the top of a ridge, and therefore it only makes sense to solve the equation away from $h^{-1} (0)$, hence we work with test functions with compact support in $\Omega \setminus h^{-1} (0)$.  This is analogous to the entropy solutions, for which we may allow $h$ to vanish on a larger subset, but then we work on $S \subset \Omega$ such that the measure of 
$h^{-1} (0)  \cap S$ vanishes.

Under certain hypotheses such as those discussed in the following section, weak solutions do exist, and the following estimates shall be useful.  

Note that equation (\ref{eq:toy2}) is the gradient flow associated to the energy functional 
$$K(f) : = \int_{\Omega} \frac{|\n f|^4}{4} h^{10/3} d\bx.$$
We first demonstrate that both the $\cL^2$ norm and the energy of weak solutions is decreasing along the gradient flow.  

\begin{lemma}
\label{lem:apriori1}
Let $H$ be a weak solution of (\ref{eq:toy2}) on $(0,T)$ for some $T>0$.  Then, both $K(H)$ and $||H||_{\cL^2 (\Omega)}$ are decreasing on $t \in (0,T)$.  
\end{lemma} 

\begin{proof} To prove that the $\cL^2$ norm decreases, we multiply both sides of (\ref{eq:toy2}) by $H$ and integrate over $\Omega,$
$$\frac{d}{dt} || H||^2 _2 = 2 \int_{\Omega} \frac{\partial H}{\partial t} H d\bx = - 2 \int_{\Omega} | \n H|^4 h^{10/3} d\bx \leq 0.$$
The second equality follows from integration by parts and the boundary conditions.  The inequality follows since the water depth $h \geq 0$.  To prove that the energy functional is decreasing, we compute its functional derivative  
$$\dot K = \int_{\Omega} \left( | \n H|^2 \n H h^{10/3} \right) \n \dot{H} d\bx = - \int_{\Omega} \n \cdot \left( |\n H|^2 \n H h^{10/3} \right) \dot{H} d\bx,$$
so that
$$D_H K =  - \n \cdot \left( |\n H|^2 \n H h^{10/3} \right).$$
The flow is defined by
$$\frac{\partial H}{\partial t} = -D_H K.$$
Consequently,
$$\frac{\partial K}{\partial t} = \int_{\Omega} D_H K \frac{\partial H}{\partial t} d\bx = \int_{\Omega} -|D_H K |^2 d\bx \leq 0.$$
\end{proof}

\begin{remark} By Proposition \ref{prop:amrt}, $\cC^\infty (\Omega)$ is dense in $\W$, so when working with weak solutions we shall integrate by parts as though we were dealing with smooth solutions.  Integration by parts holds for a sequence approximating the weak solution in $\W$, so we may integrate by parts and then take the appropriate limit. The details are straightforward.
\end{remark}

Next, we demonstrate that when weak solutions exist, then they are $\cL^2$ unique.  

\begin{theorem} \label{th:u}  Assume $F$ and $H$ are weak solutions to (\ref{eq:toy2}) with respect to the same height function $h$ with the same initial data.  Then $F$ and $H$ are equal as elements of $\cL^2(\Omega)$.
\end{theorem} 

\begin{proof}  We compute the derivative of the $L^2$ norm of $H-F$ with respect to time.
$$\frac{d}{dt} || H-F||^2 _2 = 2 \int_{\Omega} \n \cdot \left( ( \n H |\n H|^2 - \n F |\n F|^2 )h^{10/3} \right) (H-F) dx dy. $$
Integrating by parts and using the boundary conditions gives
$$ \frac{d}{dt} || H-F||^2 _2 = - 2\int_{\Omega}[ |\n H|^4 + |\n F|^4 - < \n H, \n F> (|\n H|^2 + |\n F|^2)] h^{10/3} dx dy.$$
By the point-wise Schwarz inequality applied to $< \n H, \n F>,$
\begin{equation} \label{eq:u} \frac{d}{dt} || H-F||^2 _2 \leq - 2\int_{\Omega}[ |\n H|^4 + |\n F|^4 - |\n H| | \n F| (|\n H|^2 + |\n F|^2)] h^{10/3} dx dy. \end{equation}
It is a straightforward exercise to show that for any $a, b \geq 0,$
$$a^4 + b^4 - ab(a^2 + b^2) \geq 0.$$
Consequently, the integrand in the right side of (\ref{eq:u}) is non-negative almost everywhere on $\Omega,$ which shows that
$$\frac{d}{dt} || H-F||^2 _2 \leq 0.$$
Since $H$ and $F$ have the same initial data, $||H-F||_2 = 0$ for $t=0$, which implies $||H-F||_2 \equiv 0$ for all $t \geq 0$.  This implies $H=F$ as elements of $\cL^2(\Omega)$. 
\end{proof}

Finally, we prove that all weak solutions are entropy solutions.  

\begin{prop} Given $h$ which satisfies (\ref{eq:h}) and (\ref{eq:Mucken}), any weak solution of (\ref{eq:toy2}) is also an entropy solution.
\end{prop}  

\begin{proof}
It follows from Lemma \ref{lem:apriori1} that $H \in \cL^2 (\Omega)$ for all $t \in (0, T)$, and by the definition of weak solution, $H(t) \in \W$ for all $t \in (0, T)$.  Therefore,  $T_k H(t) \in \W$ for all $k>0$.  For any $f \in \cC^\infty _0 (\Omega)$, 
$$\int_\Omega f \frac{\pa H}{\pa t} d\bx = -\int_\Omega \n f \cdot \n H |\n H|^2  h^{10/3} d \bx.$$ 
By the assumption on $h$ (\ref{eq:Mucken}), smooth functions are dense in $\W$ with respect to the associated norm. Therefore, we may take a sequence of smooth approximating functions to demonstrate (\ref{ent-sol}).  
\end{proof}  

\section{Model solutions:  mountains and ridges}
In application, one would like to model eroding landsurfaces.  Although the mathematical theory guarantees the existence of solutions, this theory alone does not produce \em model solutions \em to the equation which can be used to simulate landsurfaces.  This is one motivation for relating the partial differential equation to an optimal transport problem for the sediment.  In \S 5, we prove that certain model solutions of the equation \em predict the direction of the flow of sediment when it is optimally transported.  \em  The theoretical results can then be applied to actual modeling and simulations by distinguishing those solutions which predict the optimal flow of sediment, and therefore produce accurate long-term models for the partial differential equation.  

Simulations and observations of real landsurface shapes that retain their form for a long time but decrease in elevation were studied extensively in \cite{scaling} and \cite{stfl}.  In \cite{bms2}, \em separable solutions \em of the equations (\ref{eq:toy1}) and (\ref{eq:toy1}) were discovered; these solutions exhibit the same behavior as the simulations and observations in \cite{scaling} and \cite{stfl}.  The separable solutions that are of interest to us have the general form
\begin{equation} \label{eq:sepsol}
h(x,y, t) = h(x,y), \qquad
H(x,y,t) = H_o (x,y)T(t)
\end{equation}
where $T(t)$ is a function of time.  

\begin{prop}  Assume that for a given $h$ which satisfies (\ref{eq:h}) and (\ref{eq:Mucken}) there exists $H \in W^{1,4} _h (\Omega)$ which satisfies 
\begin{equation} \label{H-lambda} H = \lambda \n \cdot \left[ |\n H|^2 \n H h^{10/3} \right], \end{equation}
for some constant $\lambda$.  
Then, 
$$H(x,y,t) := H(x,y) T(t),$$
with 
$$T(t) = \left(a -2\lambda t \right)^{-1/2}$$
is a weak solution of (\ref{eq:toy2}) on $[0,T]$ for all $T>0$ with initial data $H$.  Moreover, if for a given $h$, $H \in W^{1,4} _h (\Omega)$, and $T \in W^{1,1} ((0, \infty))$ such that $HT$ is a weak solution to (\ref{eq:toy2}), then there exist constants $a$, $b$ such that 
$$T(t) = (a+bt)^{-1/2},$$
and $H$ satisfies (\ref{H-lambda}) with $\lambda = -b/2$.  
\end{prop} 

\begin{proof} For 
$$H(x,y,t) = H(x,y) T(t),$$
separating the time and space variables in (\ref{eq:toy2}) gives the system of equations
$$\begin{array}{lll} H(x,y) = \lambda \n \cdot \left[ |\n H|^2 \n H h^{10/3} \right]; \\ T'(t) = \lambda T(t)^3. \end{array}$$
It is straightforward to compute that $T(t) = (a+bt)^{-1/2}$ satisfies the equation with $\lambda=-b/2$.   Conversely, under the assumption that $HT$ is a weak solution to (\ref{eq:toy2}), separating variables gives the equation 
$$T'(t) = \lambda T^3 (t),$$
for some constant $\lambda$.  The solutions to this equation are of the form $T(t) = (a+bt)^{-1/2}$, and we compute that this implies $\lambda = - b/2$.    
\end{proof} 

The following separable solutions were found by the first author and studied in great detail in \cite{bms3} in the one dimensional case.

\begin{lemma}\label{lem:mountain}
Let $a$, $b$, $h_1$, $c$, $d$, and $H_1$ be constants, and assume $T = T(t)$ is a function that depends only on time.  Define
\begin{equation}
\label{eq:ridges}
\begin{array}{lll} h_o(x, y)  &= & h_1 (H_1^{1/c}+ a(x-x_0) + b(y-y_0))^d \\ H_o(x, y) &= & (H^{1/c}_1+a(x-x_0) + b(y-y_0))^{c} \\ H(x,y,t) &=& H_o (x,y) T(t),
\end{array}
\end{equation}
Then there exists a function $u$ such that
\begin{equation} \label{opt-trans-u}
\frac{\nabla H}{|\nabla H|} = \nabla u.  
\end{equation} 
\end{lemma}

\begin{proof}
The necessary and sufficient condition for the existence of a function $u$ which satisfies
$\n u = \frac{\n H}{|\n H|}$,
is
\begin{equation} \label{eq:curl} 
\nabla \times \frac{\n H}{|\n H|} = 0, \end{equation}  
which is equivalent to the following condition on the partial derivatives of $H$
\begin{equation} \label{eq:cgrad} H_{xy}(H_x^2 - H_y^2)= H_x H_y (H_{xx} - H_{yy}). \end{equation}
The rest of the proof is a computation verifying this last condition.
\end{proof}

\begin{cor} Let $h$ be given and $H \in \W$ satisfy the boundary conditions.  If there exists a separable solution $H(x,y) T(t)$ to (\ref{eq:toy2}) on $[0,T]$ for some $T>0$ with initial data $H_o$ such that $H_o$ satisfies (\ref{eq:curl}), then $H$ satisfies (\ref{eq:curl}) for all $t \in [0,T)$.  
\end{cor}

\begin{remark}
When $a$ and $b$ have opposite signs, and the sign changes across $x-x_0=y-y_0$, 
the functions $h$, $H_o$, and $H$ defined in the preceding lemma are called \em mountain ridges; \em see Figure 3. When $a$ and $b$ are both positive, and we have absolute values on $x-x_0$ and $y-y_0$,  the functions are called \em mountains; \em see Figure 4.  Lemma \ref{lem:mountain} shows that the separable solutions, that are observed both numerically and empirically, satisfy the condition (\ref{eq:curl}) that we will impose in Theorem \ref{th:ot1}.
\end{remark}
\begin{figure}
\label{ridge}
\begin{center}
\includegraphics[height=2.5in]{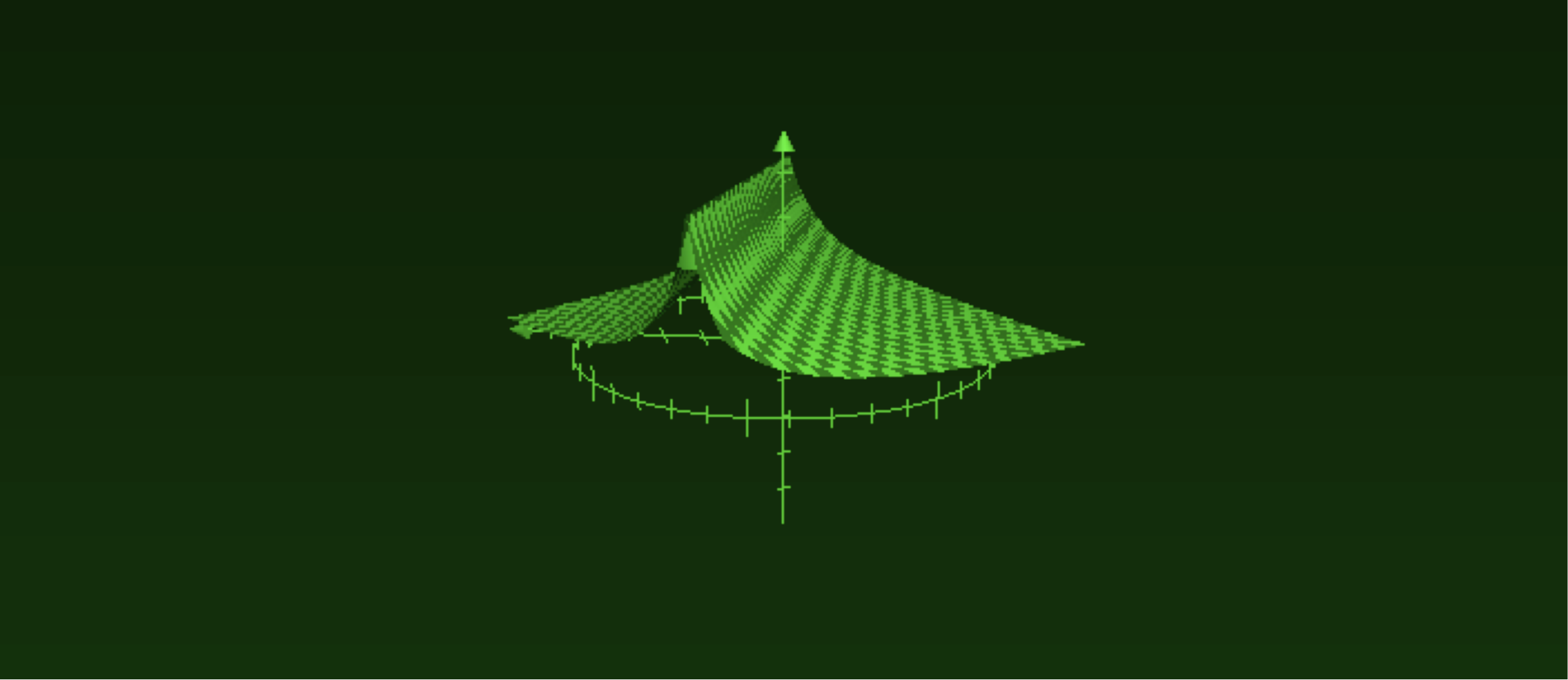}
\caption{{Ridge modeled using separable solution of the form (\ref{eq:ridges}).}}
\end{center}
\end{figure}

For the mountain and mountain ridges, if we let
\begin{equation} \label{eq:tt} T(t) = \frac{1}{\sqrt{1 + 2 r t}}, \end{equation}
then $h(x,y)$ and $H(x,y,t)$ satisfy (\ref{eq:toy2}) if the exponents $c$ and $d$ satisfy a certain relationship.  We compute this relationship to be 
\begin{equation}\label{eq:pde1} 3c - 3 + \frac{10 d}{3} = 0 \iff c = 1 - \frac{10d}{9}, \end{equation}
or
\begin{equation} \label{eq:pde2} 2c - 4 + \frac{10d}{3} = 0 \iff c = 2-\frac{5d}{3}. \end{equation}
The first condition (\ref{eq:pde1}) implies $r=0$, so that $T$ is constant, and there is no erosion.  The second condition (\ref{eq:pde2}), on the other hand, turns out to be more interesting.  This condition implies
\begin{equation}
\label{eq:r}
r = - h_1^{10/3} c^3 (a^2 + b^2)^2 (3c-3+10d/3).
\end{equation}
The constant $r$ is also related to the flux and the initial volume of sediment,
\begin{equation} \label{eq:rfv}
r = -c_r \frac{F_0}{V_0}, \quad F_0 =  \int_0^L \nabla H |\nabla H|^2 h^{10/3}(W,y,0) \cdot {\hat n} dy, \quad V_0 = \int_{\Omega} H(x, y, 0) d\bx, \end{equation}
where $F_0$ is the integration of the initial flux and $V_0$ is the initial volume of the sediment, and $c_r > 0$ is a constant.  When sediment is flowing out of $\Omega$, the integral of the flux is negative, which implies that
$$r > 0.$$
Moreover, since $h_1 > 0$, the positivity of $r$ implies that $c$ and $d$ must lie within a certain range.  In particular, we have the following.
\begin{figure}
\label{mountain}
\begin{center}
\includegraphics[height=2.5in]{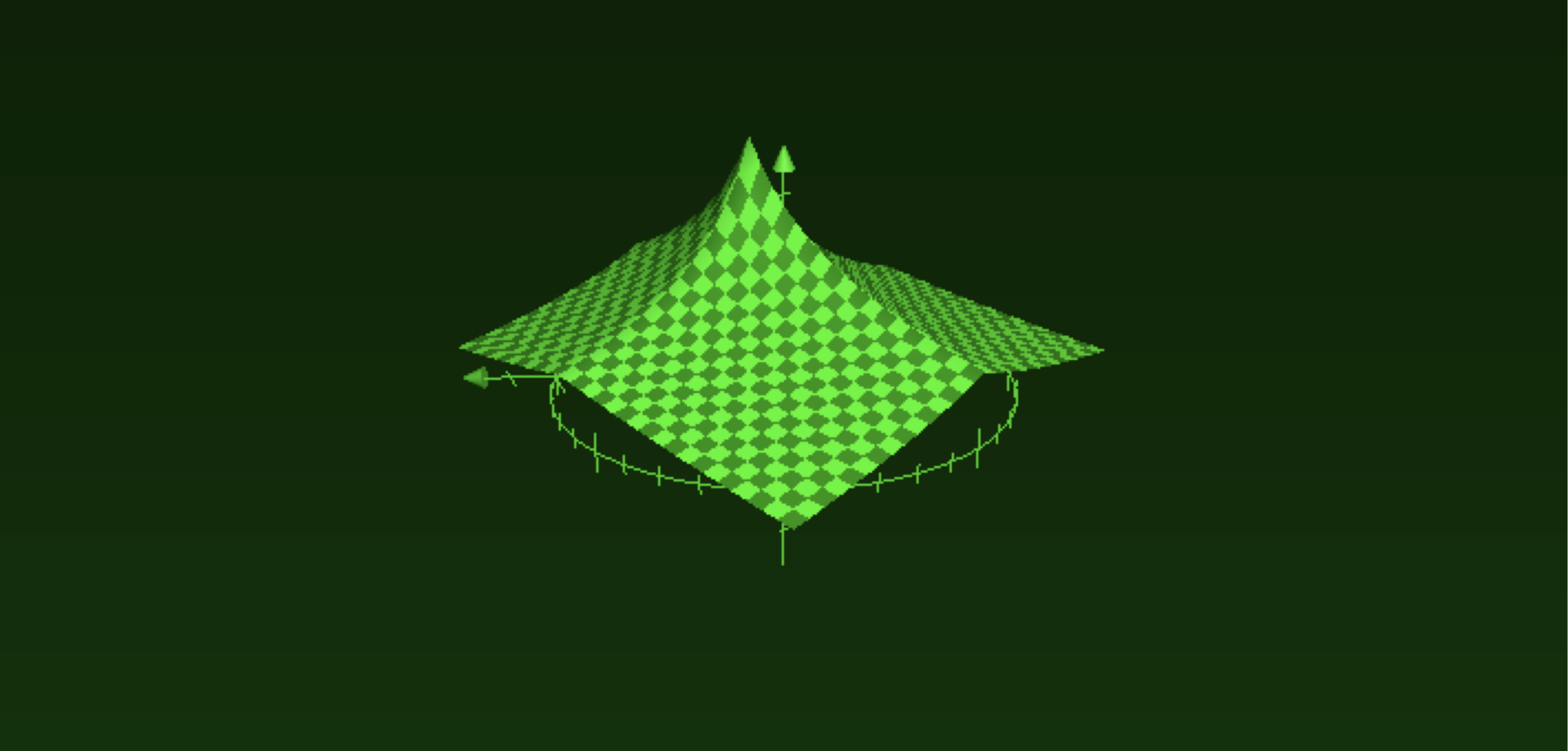}
\caption{{Mountain modeled using separable solution of the form (\ref{eq:ridges})}}
\end{center}
\end{figure}

{
\begin{lemma} \label{le:ss} 
Let 
$c$ and $d$ be constants which satisfy 
\begin{align} \label{cd-mountain}
&c = 2 - \frac{5d}{3}, \nonumber \\
& d < \frac{9}{10}. 
\end{align}                             
Then, there exist weak solutions to (\ref{eq:toy2}) for functions $h=h_0$ which are piecewise defined by (\ref{eq:ridges}) and such that $h$ satisfies (\ref{eq:boundary}).  Moreover, these solutions are also piecewise given by functions $H(x,y,t)$ which are piecewise defined by (\ref{eq:ridges}) and satisfy (\ref{eq:curl}).  
\end{lemma} 

\begin{proof}   The function $h_0$ and the constants are chosen so that $h_0 \in \cL^\infty (\Omega)$, satisfies  the boundary conditions and (\ref{eq:Mucken}).  This is straightforward.  Next, 
the functions $H$ and $h$ are piecewise defined so that $H$ is continuous, and $\n H$ is continuous on $\Omega \setminus R$ where $R$ is a piecewise linear curve which corresponds physically to the top of a mountain ridge.  The preceding direct calculations show that the function $H$ defined according to $h_0$ as in (\ref{eq:ridges}) is a strong solution to (\ref{eq:toy2}) on $\Omega \setminus R$.  We then define $h$ to vanish along $R$.  This corresponds to setting the constant $h_1 = 0$ along $R$, and consequently, $h$ is not continuous along $R$.  In the mathematical model of a linear ridge, the top of the ridge has Lebesgue measure zero in two dimensions, so any water which hits the top of the ridge must fall to either side, hence the water height in the pure mathematical model vanishes there.  In general, the probability that water falls on a specific point or set of measure zero may reasonably be zero.   In the numerical simulations $h_1$ is taken to be quite small, so based upon all these considerations, it is not an unreasonable mathematical simplification to assume $h^{-1} (0) = R$.  Then since $H$ is a strong solution on $\Omega \setminus R$, when we integrate against a test function $\phi \in \cC^\infty _0 (\Omega \setminus h^{-1}(0))$, we may integrate by parts and both the interior term vanishes (since $H$ is a strong solution there), and the boundary terms vanish since $\phi \in \cC^\infty _0 (\Omega \setminus h^{-1} (0)$, so it follows from the definition that $H$ is a weak solution to (\ref{eq:toy2}). It then follows from Theorem 3 that $H$ is the unique weak solution for this initial data, and it follows from Proposition 2 and Theorem 1 that $H$ is also unique as an entropy solution.  
\end{proof}

\begin{remark}  More generally, if $h^{-1} (0)$ were to have positive measure, then we may solve (\ref{eq:toy2}) over a smooth subset $S \subset \Omega$ such that $h$ is positive almost everywhere on $S$; c.f. \S 3 in \cite{amrt3}.  By the uniqueness of solutions, the correspondingly defined $H$ (\ref{eq:ridges}) is then the solution of the restriction of the PDE to $S$.  \end{remark}  

There also exist solutions of (\ref{eq:toy2}) which correspond to the Barenblatt solution \cite{bar},  \cite {pra} of the porous medium equation.  Define the {\em collapsing hills}, see Figure 5,  
\begin{align} \label{chill} 
& h(x,y,t) = h_1\left[(x-x_0)^2 + (y-y_0)^2)\right]^d (1+rt)^{\gamma} \nonumber  \\ & H(x,y,t)= H_1 \left[(x-x_0)^2 + (y-y_0)^2)\right]^c (1+rt)^\beta,
\end{align} 
where $h_1$, $H_1$, $r$, $\beta$, $\gamma$, $c$, and $d$  are constants. Again we may assume without loss of generality that $(x_0, y_0) = (0,0)$. Then, the collapsing hills function satisfies $H(x,y,t) = H(y,x,t)$ which immediately implies (\ref{eq:cgrad}). By a calculation similar to that for the Barenblatt solution of the porous medium equation \cite{bar}, \cite{pra}, if the constants satisfy certain constraints, then the collapsing hills are a strong solution of (\ref{eq:toy2}). A straightforward calculation shows that the constants must satisfy 
$$2\beta + 1 + 10\gamma/3 = 0 \implies \gamma = - \frac{3}{10} \left( 1+ 2\beta \right),$$
$$2c-1+10d/3 = 0 \implies d = \frac{3}{10} \left( 1-2c \right), \quad (3c-1+10d/3) = c, $$
$$\beta r = 16 H_1 ^2 c^4.$$  
For these solutions, both $\beta$ and $\gamma$ are negative, which implies that $\beta \in (-1/2, 0)$ and $\gamma \in (-3/10, 0)$.  The collapsing hill function is a strong solution to (\ref{eq:toy2}) under these conditions, and as with the mountain ridge functions, it may be piecewise defined to ensure the boundary conditions are satisfied.  

\begin{figure}
\label{mountain}
\begin{center}
\includegraphics[height=2.5in]{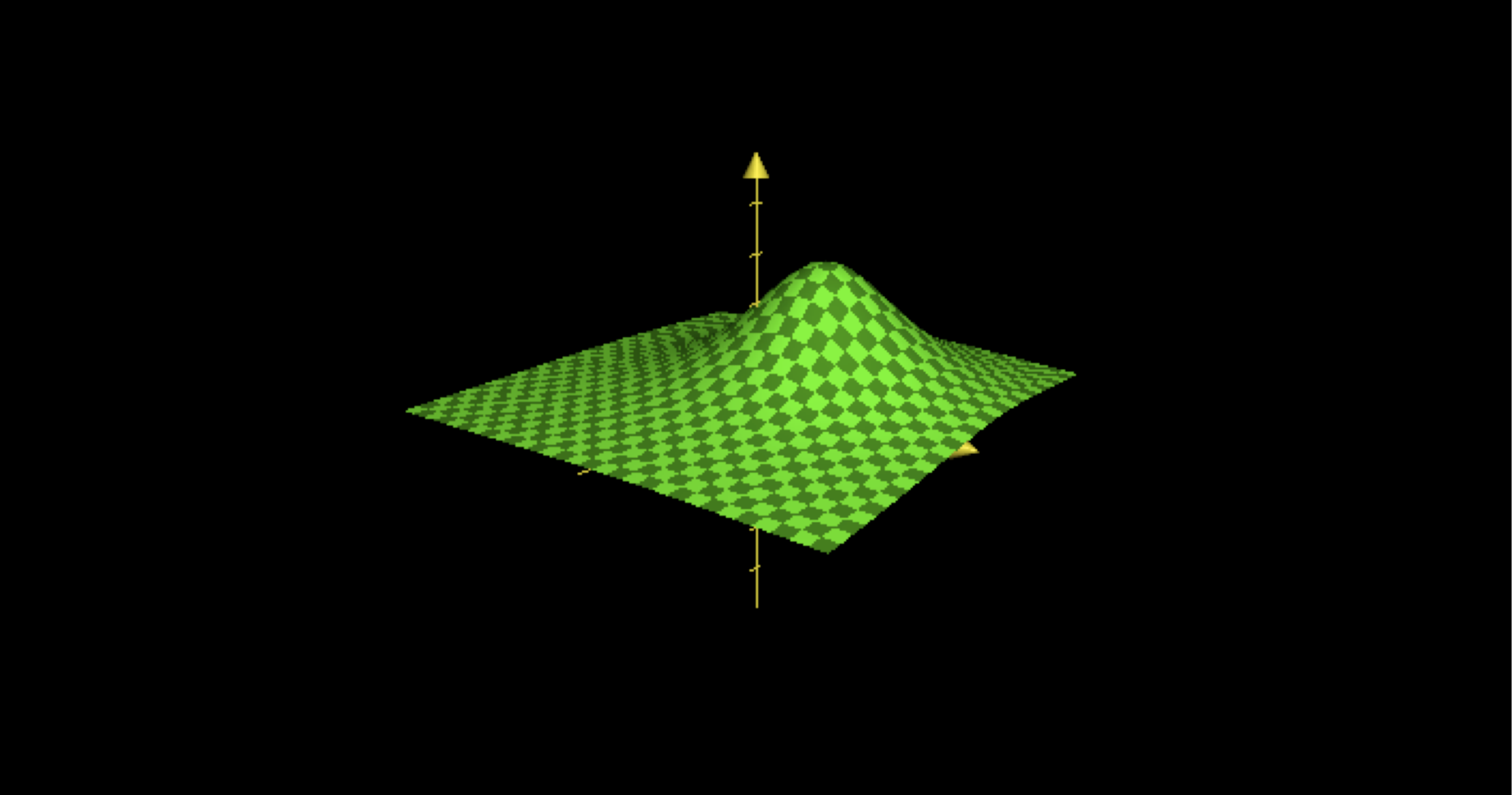}
\caption{Collapsing hill modelled using separable solution of the form (4.13).}
\end{center}
\end{figure}

We are most interested in the mountain ridges, because they are observed both empirically and in simulations for significant time intervals; see \cite{scaling} and \cite{stfl}.  The empirically observed mountain ridges are in fact more complicated than the ridges modeled by our mountain ridge functions. The observed mountain ridges are actually chains of pieces or slices defined by these ridge functions and linked together. The ridge lines form piecewise linear crests, see Figure 2 (bottom). In the limit of such chains of convex pieces the top of the mountain ridge can even form a fractal curve; see \cite{scaling} and \cite{stfl} for figures of simulations of such ridges.  The mountains are only observed for much shorter times in the simulations.  They occur at the boundary and then usually for relatively short times.  The collapsing hills are only observed briefly at the very end of simulations when the surface quickly collapses to a flat plain.    Only mountains that anchor stable mountain ridges at the boundary persist for  long times.   As time becomes large, it is observed that all solutions tend toward these separable solutions, that is a pattern of valleys separated by convex mountain ridges; see \cite{scaling} and \cite{stfl}.  This is what we recognize as ``the landscape.'' Based on our work and the related work of Otto \cite{ott}, we expect that all solutions of (\ref{eq:toy2}) tend toward these separable mountain ridges; further discussion of this is postponed to \S 6.  

\section{Optimal Transportation}
We recall the general setup of optimal transport problems \cite{monge}.  Let $\mu$ and $\nu$ be non-negative Radon measures with (respectively) compact supports $U, V \subset \R^n$ satisfying,
\begin{equation} \label{eq:balance} \int_{U} d\mu = \int_{V} d\nu.  \end{equation}
A map $s: U \to V$ pushes $\mu$ onto $\nu$, and we write $s_\# (\mu) = \nu$ if $s$ is Borel measurable and for any Borel set $E \subset V$,
\begin{equation} \label{eq:preserve} \int_{s^{-1}(E)} d\mu = \int_E d\nu. \end{equation}
Associated to the optimal transport problem is a cost function which is typically given by
\begin{equation} \label{eq:cost} C(s) := \int_U c(\bx, s(\bx)) d\mu(\bx), \qquad c(\bx, \by) := \frac{|\bx - \by|^p}{p}, \end{equation}
where $p \geq 1$ is fixed.  Monge's original problem, with $p=1$, is in fact more difficult than the problem with $p > 1$; in this work, we investigate the case $p=1$.  A general optimal transport  problem is,
\begin{equation} \label{monge} \textrm{ \em Does there exist $s:U \to V$ which minimizes $C$ with $s_\# (\mu) = \nu$? \em} \end{equation}
If it exists, such a map $s$ is called an ``optimal mass reallocation plan," or an ``optimal mass transport plan.''

\subsection{An optimal transport problem for the flow of sediment}
We  consider an ``instantaneous optimal transport problem'' for the sediment similar to the equation used to model sand cone dynamics in \cite{evans} \S 11.  The equation considered in that work is
$$\left\{ \begin{array}{ll}  f - u_t \in I_{\infty} [u] & (t > 0) \\ u = 0 & (t=0) \end{array} \right.$$
where $I_{\infty} [u]$ is a certain functional defined in \cite{evans} (9.16) and (9.17).  The physical interpretation of such an instantaneous optimal transport problem is that at each moment in time, the mass $d\mu^+ = f^+ (\cdot, t)dx$ is instantly and optimally transported downhill by the potential $u(\cdot, t)$ into the mass $d\mu^- = u_t (\cdot, t)  dy$.  In other words, the height function of the sandpile is also the potential generating the optimal transport problem $u_tdx \mapsto f^+ dy$.  To study the local behavior of the flow of sediment under erosion, it is then natural to introduce a similar instantaneous optimal transport problem.

By the divergence theorem and the boundary conditions
\begin{equation} \label{eq:massbal} \bar{F}_{\Omega} := \int_\Omega \frac{\pa H}{\pa t}d\bx = \int_0^L \nabla H |\nabla H|^2 h^{10/3}(W,y,t) \cdot  n dy. \end{equation}
We make the natural assumption that the sediment is flowing out of the region $\Omega$ into the lake or river which meets the $\{x=W\}$ boundary of $\Omega$, so that
$$\bar{F}_{\Omega} < 0.$$
We formulate the optimal transport problem using the sediment flux instead of the mass. The problem then becomes an optimal transport problem of the sediment fluxes. This is however equivalent to the optimal transport problem of the masses transported by the sediment fluxes in a small time interval as will be illustrated below.

Define the measures $\mu$ and $\nu$ with support on $\Omega$,
\begin{equation} \label{eq:mn} d \mu := -\frac{\partial H}{\partial t} (\bx, t)d\bx=: f^+(\bx)  d\bx  , \qquad d \nu := -Fd\bx  =: f^- (\bx) d\bx. \end{equation}
where
$$F := \bar{F}_\Omega/|\Omega|,$$
and $|\Omega|$ denotes the area of $\Omega$. The density $F$ is constant on $\Omega$ but this is the result of averaging the non-constant line density on the boundary in (\ref{eq:massbal}) and spreading it uniformly over $\Omega$.
We want to know if this formulation of the optimal transport amounts to nature taking
mounds of dirt (mountains) and dumping them in the ocean. To see this we rewrite the balance equation (\ref{eq:balance}) as
\[
\int_\Omega -\frac{\partial H}{\partial t} d{\bf x} = -\int_0^L \nabla H |\nabla H|^2 h^{10/3}(W,y,t) \cdot  n dy.
 \]
If we integrate this equality over a small time time interval, we get
\[
\int_\Omega (H_0(x,y)-H(x,y,t)) d{\bf x} =  -\int_0^t {\bar F}_\Omega(t)dt.
\]
Thus the dirt removed from the surface equals the cumulative flux that exited the lower boundary in the time interval $[0,t]$.
Here we have formulated the problem in terms of an area density being transported to a line density. However, it is more convenient to be able to integrate over the same domain on both sides of (\ref{eq:balance}) and therefore we spread the transported sediment again uniformly over $\Omega$ in (\ref{eq:mn}) for convenience of the exposition.

We make the natural assumption that the landsurface is eroding:  that its height is decreasing
\begin{equation} \label{eq:out} \frac{\pa H}{\pa t} \leq 0 \quad \textrm{ a. e. on } \Omega. \end{equation}
Under these assumptions, the measures are non-negative.  The physical interpretation of the mass reallocation problem (\ref{monge}) for $\mu \mapsto \nu$, is that at time $t_0$ the sediment is instantly and optimally transported.
In other words the sediment flux $-d \nu := Fd\bx$ is equal to the rate of decrease in the height of the water surface $-d \mu := \frac{\partial H}{\partial t} (\bx, t) d\bx$. We will show that if this transport implemented by the sediment flow is in the direction of the negative surface gradient $-\nabla H$, then it is in fact optimal.


\subsection{Proof of Theorem \ref{th:ot1}} 
Since $H$ is a weak solution of (\ref{eq:toy2}), $f^{\pm} \in \cL^1(\Omega)$.  By definition of $\mu$ and $\nu$ and (\ref{eq:massbal}), the mass balancing condition
$$\int_{\Omega} d\mu = \int_{\Omega} d\nu$$
is satisfied.  
Moreover, the measures are by hypothesis non-negative and absolutely continuous with respect to Lebesgue measure
$$d\mu, d\nu  << d\bx.$$
The existence of the optimal mass reallocation plan $s$ and a function $u$ so that $s$ and $u$ satisfy (\ref{eq:su}) is well know; see for example \cite{villani}, \cite{tw}, and \cite{evans}. This proves the first statement in the theorem.  Demonstrating (\ref{eq:grad}) under the assumption (\ref{eq:curl}) will require a bit more work.

The main idea in the proof of the optimal transport is to carefully analyze Kantorovich's dual maximization problem, namely to maximize
$$\cK [u, v] := \int_{\Omega} u(\bx) d\mu (\bx) + \int_{\Omega} v(\bx) d\nu (\bx)$$
subject to the constraint
$$u(\bx) + v(\by) \leq c(\bx, \by) \textrm{ for } \bx, \by \in \Omega.$$
Since we are working with Monge's original cost function, $c(x,y) = |x-y|$, by \cite{evans} Lemma 9.1 we may assume that
$$u = - v.$$
In fact, \cite{evans} requires additional regularity on $f^{\pm}$, but this is not necessary as demonstrated in \cite{tw}.  The constraint may then be reformulated to
\begin{equation} \label{eq:lipc} | v(\bx) - v(\by)| \leq |\bx - \by| \textrm{ almost everywhere on } \Omega. \end{equation}
With this simplification, the dual problem is to maximize
$$\cK (v) := \int_{\Omega} v(\bx) (f^+ - f^-) d\bx,$$
subject to the Lipschitz constraint (\ref{eq:lipc}).

In the definition of weak solution, we may integrate by parts for any smooth test function compactly supported in $\Omega \setminus h^{-1}(0)$.  Moreover, $h^{-1} (0)$ is the finite union of piecewise smooth curves, and therefore such test functions are $\cL^2$ dense in $\cL^2 (\Omega)$.  Within $\Omega$ we can approximate any arbitrary $v \in \cL^2(\Omega)$ by test functions, and by the boundary conditions, since $v$ need not vanish on the boundary, we have 
$$\int_{\Omega} v(f^+ - f^-)d\bx = I + II + III,$$
where 
$$I = \int_{\Omega} \langle \n v, \n H \rangle |\n H|^2 h^{10/3} d\bx,$$
$$II = - \int_{y=0} ^L \n H (W,y,t) \cdot n |\n H|^2 h^{10/3} v dy, \quad III = \bar{F}_{\Omega} \int_{\Omega} v d\bx.$$
Since the integrands in $I$ and $II$ both vanish at points where $\n H$ vanishes, and since $\n H$ is defined a.e. on $\Omega$, we shall maximize $\cK (v)$ if we maximize 
$$I' + II+ III,$$
where 
$$I' = \int_{\Omega'} \langle \n v, \n H \rangle |\n H|^2 h^{10/3} d\bx, \quad \Omega' := \{ \bx \in \Omega: \n H \textrm{ is defined and nonzero}\}.$$

By the pointwise Schwarz inequality,
\begin{equation} \label{est:v} |\langle \n v, \n H \rangle| \leq |\n v| |\n H|, \end{equation}
with equality if and only if $\n v$ is a scalar multiple of $\n H$ so that $\n v = c \n H$.  The only scalar multiples consistent with the Lipschitz constraint are $c = \pm \frac{1}{|\n H|}$.  Thus, for any test function $v$ satisfying the Lipschitz constraint,
$$\int_{\Omega}\langle \n v, \n H \rangle |\n H|^2 h^{10/3} d\bx \leq \int_{\Omega}\langle \n u, \n H \rangle |\n H|^2 h^{10/3} d\bx,$$
where $u$ is defined to satisfy (\ref{opt-trans-u}).   We conclude that the maximizer of $\cK$ is achieved by $u$ which satisfies (\ref{opt-trans-u}) and maximizes $II$ and $III$, noting that these conditions are independent of the condition on the gradient of $u$.   By \cite{tw} Theorem 3.1, there exists an optimal mass reallocation plan $s$ such that
$$\frac{s(\bx) - \bx}{|s(\bx) - \bx|} = - \n u = - \frac{\n H}{|\n H|}, \quad \textrm{a.e. on $\Omega$ where $\n H$ is defined and non-zero.}$$ 
\qed

The physical interpretation of $\n H (\bx, t) = 0$ is that the point $\bx$ lies at the top of a mountain; such points empirically form a set of measure zero.  Since the sediment flows in the direction of $-\n H$, our result shows that the direction of the sediment flow according to the solution of (\ref{eq:toy2}) is identical to the direction of the instantaneous optimal mass reallocation plan almost everywhere on $\Omega$.  Therefore, the direction in which the sediment flows according to (\ref{eq:toy2}) is optimal when the landsurface evolves according to the separable solutions in \S 4.  We expect that in general, solutions to (\ref{eq:toy2}) evolve over time toward certain optimal solutions; this is discussed in the following section.  

\section{Gradient flows and long time asymptotics}
We have focused on the local properties of the optimal mass reallocation plan and its relationship to the local properties of the sediment flow.  This is related to the porous medium equation
\begin{equation} \label{eq:ott} \frac{\partial \rho}{\partial t} = \n ^2 \rho^m, \end{equation}
where $\rho \geq 0$ is a time dependent density function on $\R^n$, and $m \geq 1$.  When $m>1$, this represents so-called ``slow diffusion;'' $m < 1$ is called fast diffusion.  In \cite{ott}, the exponent satisfies $m \geq 1 - \frac{1}{n}$ and $m > \frac{n}{n+2}$.  In an appropriate weak setting, similar to ours, the Cauchy problem for (\ref{eq:ott}) is well posed.  Then, (\ref{eq:ott}) defines an evolution of densities on $\R^n$.  
Expressing the porous medium equation as the gradient flow
$$\frac{d}{dt} E(\rho) = - g_{\rho} \left( \frac{d\rho}{dt}, \frac{d \rho}{dt} \right),$$
separates the energetics and kinetics:  the energetics are represented by the functional $E$ on the state space $M$ while the kinetics endow the state space with Riemannian geometry via the metric tensor $g$.  This state space $M$ naturally carries the Wasserstein distance.  The main results of \cite{ott} demonstrate that the density gradient flow converges, at a certain rate made explicit in the paper,  to the Barenblatt solution, which minimizes the energy functional.  This is equivalently described on the state space:  the gradient flow tends towards the optimal measure.  Thus, \cite{ott} establishes a connection between the space of probability measures equipped with the Wasserstein metric and the long time behavior of solutions to the porous medium equation.

The setting in \cite{ott} does not immediately apply to our problem. Both the weight function and the mixed boundary conditions appear to influence the asymptotics, making them different from \cite{ott}.  The Barenblatt solution plays the main role in \cite{ott}, but in our case the collapsing hill (\ref{chill}), that is the analog of the Barenblatt solution, is not observed to be the main actor in the asymptotics. Instead that role is played by the mountain ridge functions in Lemma \ref{lem:mountain}. Nevertheless the structure in \cite{ott} appears adaptable to our case, and one should be able to use the Wasserstein metric to describe how our general solutions approach the optimal metric, given by the mountain ridges, as time tends to infinity.   It would be interesting to numerically simulate both the equation (\ref{eq:toy2}) and the optimal transportation problem and compare the direction of $\n H$ and the direction of the optimal transportation over time.   Even more intriguing is the question of whether the stochastic approach \cite{stfl} can be formulated on the space $M$ where the probability measures and the Wasserstein metric live? These questions will be the subject of future work.

\end{document}